\documentclass[11pt]{amsart}
\usepackage{amsmath,amsthm,amsfonts,amssymb,verbatim,eucal}
\usepackage{ graphics, color}
\usepackage[all,line]{xy}
\usepackage{enumerate}

\numberwithin{equation}{section}

\newtheorem{thm}[equation]{Theorem} 
\newtheorem{prop}[equation]{Proposition}
\newtheorem{lemma}[equation]{Lemma} 
\newtheorem{cor}[equation]{Corollary}
\newtheorem{example}[equation]{Example}
\newtheorem{remark}[equation]{Remark}
\newtheorem{definition}[equation]{Definition}

\newenvironment{ex}{\begin{example}\rm}{\end{example}}
\newenvironment{rem}{\begin{remark}\rm}{\end{remark}}

\newcommand{\diag}{\text{diag}}

\newcommand{\NN}{\mathbb N}

\renewcommand{\k}{{{\mathbb K}}}
\newcommand{\kG}{{{\mathbb K}G}}

\newcommand{\Z}{{\mathbb Z}}
\newcommand{\N}{{\mathbb N}}
\newcommand{\id}{\mbox{\rm id\,}}      
\newcommand{\Hom}{\mbox{\rm Hom\,}}

\newcommand{\C}{\mathcal C}
\newcommand{\U}{\mathcal{U}}

\newcommand{\iso}{\cong}

\newcommand{\ot}{\otimes}

\newcommand{\cH}{\mathcal{H}}
  
\newcommand{\ZZ}{\mathbb{Z}}

\newcommand{\tensor}{\otimes}

\newcommand{\Hqk}{{\mathcal {H}}_{{\bf{q}},\kappa}}

\begin{document}

\title[Color Lie rings]
{Color Lie rings and PBW deformations\\ of skew group algebras}

\date{January 26, 2018}

\author{S. Fryer}
\address{Department of Mathematics, University of California Santa Barbara,
CA 93106, USA}
\email{sianfryer@math.ucsb.edu} 
\author{T. Kanstrup}
\address{Max Planck Institute for Mathematics, Vivatsgasse 7, 
53111 Bonn, Germany}
\email{tina.kanstrup@mail.dk}
\author{E. Kirkman}
\address{Department of Mathematics, PO\ Box 7388, Wake Forest
University, Winston-Salem, NC 27109, USA}\email{kirkman@wfu.edu} 
\author{A.V.\ Shepler}
\address{Department of Mathematics, University of North Texas,
Denton, TX 76203, USA}
\email{ashepler@unt.edu}
\author{S.\  Witherspoon}
\address{Department of Mathematics\\Texas A\&M University\\
College Station, TX 77843, USA}\email{sjw@math.tamu.edu}


\maketitle

\begin{abstract}
We investigate color Lie rings over finite group algebras and
their universal enveloping algebras.
We exhibit these universal enveloping algebras as PBW deformations of skew group algebras:
Every color Lie ring over a finite group algebra
with a particular Yetter-Drinfeld structure has universal enveloping
algebra that is a quantum Drinfeld orbifold algebra.
Conversely, every quantum Drinfeld orbifold algebra of
a particular type arising from the action of an abelian group
is the universal enveloping algebra of some color Lie ring
over the group algebra.
One consequence is that these quantum Drinfeld orbifold algebras
are braided Hopf algebras.
\end{abstract}


\section{Introduction}

We examine color Lie rings arising from finite groups
acting linearly on finite dimensional vector spaces.
We show that color Lie rings with Yetter-Drinfeld structures
have universal enveloping
algebras that are PBW deformations of skew group rings
(i.e., semidirect product algebras) under
a mild condition on their colors.
We consider only abelian groups since antisymmetry of the Lie 
bracket implies that any grading group is abelian. 
Specifically, 
we consider a color Lie ring 
arising from a finite abelian group
$G$ acting diagonally on a finite dimensional vector space 
$V$ over a field $\k$, and we exhibit
its universal enveloping algebra
as a quantum Drinfeld orbifold algebra.

Quantum Drinfeld orbifold algebras~\cite{Shroff} generally
are filtered PBW deformations
of skew group algebras $S_{\bf {q}}(V)\rtimes G$
formed by finite groups $G$ acting on quantum symmetric algebras
(skew polynomial rings) $S_{\bf {q}}(V)$ by graded automorphisms.
They include as special cases, for example, the
\begin{itemize}
  \item
universal enveloping algebras of Lie algebras
(with ${\bf q}$ and $G$ both trivial),
  \item
symplectic reflection algebras
(with ${\bf q}$ trivial and $G$ symplectic)~\cite{EtingofGinzburg},
\item
  graded affine Hecke algebras and
  generalizations to quantum polynomial rings
  (with ${\bf q}$ and
  $G$ nontrivial)~\cite{LevandovskyyShepler,Lusztig},
and
\item
generalized universal enveloping algebras of color Lie algebras
(with $G$ trivial)~\cite{ShroffWitherspoon}.
\end{itemize}
When the group $G$ is nontrivial,
previous work has concentrated on the case
of nonhomogeneous quadratic algebras
whose defining relations set
commutators of vector space elements equal to elements in the group
algebra $\k G$, i.e., $[v, v]\in \k G$ in the algebra.
Alternatively, here we require
commutators of vector space elements to equal
elements in the space  $V\ot \k G$; our algebras then will be 
direct generalizations of universal enveloping algebras of Lie
algebras (the origin of the term ``PBW deformation'').

We examine closely this latter type of PBW deformation
of $S_{\bf{q}}(V)\rtimes G$ in characteristic~0, searching for  Lie structures.
In Section~\ref{sec:tau}, starting with an abelian group $A$, an antisymmetric  bicharacter $\epsilon$ on $A$, and a ring $R$, we define the notion of an $(A,\epsilon)$-{\em color Lie ring} over $R$ and its universal enveloping algebra.  
When a finite abelian group $G$ acts {\em diagonally}
on $V=\k^n$ with basis $v_1,\ldots, v_n$,
the space 
$$
L=V\ot \kG = \bigoplus_{i, g} \k (v_i\ot g)
$$
is naturally graded as a vector space
by the abelian group $A=\ZZ^n\times G$.
We define an antisymmetric bicharacter on $A$
and show that $L = V \otimes \k G$ is a Yetter-Drinfeld module over $\kG$,
setting the stage for the {\em Yetter-Drinfeld color Lie rings} 
defined in Section~\ref{sec:tau} and featured in later sections.

In Section~\ref{sec:qdoa},
we define deformations of $S_{\bf{q}}(V)\rtimes G$
corresponding to parameter functions  
$\kappa: V \times V \rightarrow V \otimes \kG$: 
\begin{equation*}
\mathcal{H}_{\bf q, \kappa}:=
(T(V)\rtimes G)/\big(v_i v_j - q_{ij} v_j v_i-\kappa(v_i,v_j):
1\leq i,j\leq n\big).
\end{equation*}
We analyze
conditions for $\mathcal{H}_{\bf q, \kappa}$ to satisfy the PBW property,
i.e., to be a PBW deformation of $S_{{\bf q}}(V)\rtimes G$.
Then we present several examples that 
meet these conditions on $\kappa$, the group action, and the parameter set ${\bf{q}} =\{q_{ij}\}$.
Under these conditions, we call $\mathcal{H}_{\bf{q},\kappa}$ a 
{\em quantum Drinfeld orbifold algebra}. 
In Sections~\ref{FirstMainThm} and~\ref{gradedcolorLierings},
we determine conditions 
on the algebra $\mathcal{H}_{\bf q, \kappa}$
under which an underlying 
subspace $L 
=V \ot \kG$
is a color Lie ring over $\kG$, generalizing a result 
in~\cite{ShroffWitherspoon} to nontrivial $G$.
We show that $\mathcal{H}_{\bf q, \kappa}$ is isomorphic to the universal enveloping algebra $\U(L)$ and, in Section~\ref{sec:cat-th},
exhibit a Hopf algebra structure on 
$\mathcal{H}_{\bf q, \kappa}$.

\vspace{4ex}

\noindent
    {\bf Example}.
    Let $V=\k^3$ with basis $v_1,v_2,v_3$
    and 
    suppose
    $\k$ contains a primitive $m$-th root of unity $q$
     for some $m\geq 2$.
     Let $g$ be the diagonal matrix $\diag(q,1,1)$
     and set $G=\left<g\right>$.
     Then $L = V\ot \kG$ is a {\em color Lie ring} over $\kG$ with
     brackets determined by
\[
    [v_1\ot 1,v_2\ot 1] = v_1 \ot g , \ \ \ 
   [v_1\ot 1,v_3\ot 1]=0, \ \ \ [v_2\ot 1,v_3\ot 1] = 0 \, .
\]
The universal enveloping algebra $\mathcal{U}(L)$ is a PBW deformation
of $S_{{\bf q}}(V)\rtimes G$ for a choice of parameter set ${{\bf q}}$. 
Details and more examples are in Section~\ref{sec:qdoa}.

\vspace{4ex}

Color Lie rings carry an additional grading by $\ZZ/2\ZZ$,
and we may speak of their positive and negative
parts with respect to this grading.
In Section~\ref{SecondMainThm}, we determine
how color Lie rings with only positive part
define deformations of the noncommutative algebra
$S_{\bf q}(V)\rtimes G$.
The following  result is our main Theorem~\ref{thm:main-recap}.

\medskip

\noindent
{\bf Theorem.} 
{\em 
Let $G$ be a finite group of diagonal matrices acting
on $V=\k^n$.
\begin{itemize}
\item[(i)]
The universal enveloping algebra
of any Yetter-Drinfeld color Lie ring
with purely positive part
is isomorphic to a
quantum Drinfeld orbifold algebra $\cH_{{\bf q},\kappa}$ for 
some parameters ${\bf q}$ and $\kappa$.
\vspace{1ex}
\item[(ii)]
Any quantum Drinfeld orbifold algebra $\cH_{{\bf q},\kappa}$ with
$\kappa: V \times V \rightarrow V \otimes \kG$
satisfying an additional condition 
is isomorphic to the universal enveloping algebra $\mathcal{U}(L)$ of some
Yetter-Drinfeld color Lie ring $L$.
\end{itemize}
}

\medskip


Part (i) shows that certain color Lie rings $L$ over $\kG$ have
universal enveloping algebras $\U(L)$
defining PBW deformations of $S_{\bf q}(V)\rtimes G$. 
This immediately implies a PBW theorem for these algebras. 
Part~(ii) exhibits many quantum Drinfeld orbifold algebras
as the universal enveloping algebras of color Lie rings.
In Section~\ref{sec:cat-th}, we regard color Lie rings $L$  
as Lie algebras in certain braided symmetric monoidal categories, 
and thus
their universal enveloping algebras $\U(L)$
inherit the structure of {\em braided Hopf algebras}. 
An immediate consequence is that certain
quantum Drinfeld orbifold algebras are
braided Hopf algebras.
We  obtain the following corollaries
(Corollary~\ref{cor:Hopf-alt}
and~\ref{lastcor}).

\medskip

\noindent
{\bf Corollary.} {\em 
 A quantum Drinfeld orbifold algebra $\cH_{{\bf q},\kappa}$
with $G$ abelian acting diagonally 
on $V$ and $\kappa: V \times V \rightarrow V \otimes \kG$
satisfying an extra condition 
is a braided Hopf algebra.}

\medskip

\noindent
{\bf Corollary.}
{\em 
Let $L$ be a Yetter-Drinfeld color Lie ring
that is purely positive.
Then its universal enveloping algebra $\U(L)$ has the PBW property.}

\medskip

Throughout, $\k$ will be a field of characteristic~0
unless stated otherwise. 
An unadorned tensor symbol between modules is understood to be
a tensor product over $\k$,
that is, $\otimes =\otimes_{\k}$. 
The symbol $\N$ denotes the set of natural numbers, 
which we assume includes $0$.
In any algebra containing a group algebra $\kG$, we
identify the unity of the field $\k$
with the identity of $G$: $1=1_\k=1_G$.

\section{Color Lie rings and
universal enveloping algebras}
\label{sec:tau}

In this section, we introduce color Lie rings and their universal
enveloping algebras. 

\subsection*{Gradings}
Let $A$ be an abelian group.
A $\k$-algebra $R$ is $A$-graded if 
$R=\oplus_{a\in A} R_a$ as a vector space and 
$R_aR_b \subset R_{ab}$
for all $a,b\in A$.
An {\em $A$-graded $R$-bimodule} is an $R$-bimodule $L$ together with a 
decomposition $L = \oplus_{a\in A} L_a$
as a direct sum of vector spaces $L_a$
such that $R_a L_b R_c \subset L_{abc}$ for all $a,b,c\in A$. 
If $0\neq x\in L_a$, we write $|x|=a$ and say $x$
is {\em $A$-homogeneous}. We define $|0|$ to be the identity $1_A$ of $A$ 
(note $0\in L_a$ for all $a\in A$).

If an abelian group $A$ is grading a group algebra $R=\k G$,
then we make the additional assumption
that each $g$ in $G$ is $A$-homogeneous and we 
say that {\em $G$ is graded by $A$}. 

\subsection*{Color Lie rings}
A function $\epsilon: A\times A\rightarrow \k^*$
is a {\em bicharacter} on an abelian group 
$A$ if
$$\epsilon(a,bc)=\epsilon(a,b)\epsilon(a,c)
\quad\text{and}\quad
\epsilon(ab,c)=\epsilon(a,c)\epsilon(b,c)\
\quad\text{ for all }
a,b,c\in A\ . $$
A bicharacter $\epsilon$ is {\em antisymmetric} if 
$\epsilon(b,a)= \epsilon(a,b)^{-1}$ for all $a,b\in A$. 
Note that any antisymmetric bicharacter $\epsilon$
satisfies (for all $a,b$ in $A$)
$$\epsilon(a,a)=\pm1,\quad \epsilon(a,1_A)=\epsilon(1_A,a)=1,
\quad\epsilon(a,b^{-1})=\epsilon(a^{-1},b)=(\epsilon(a,b))^{-1}\  .$$
\begin{definition}\label{def:colorLie}
{\em 
Let $A$ be an abelian group and let $\epsilon$ be an antisymmetric bicharacter on $A$. 
Let $R$ be an $A$-graded $\k$-algebra. 
An {\em $(A,\epsilon)$-color Lie ring}  over $R$  is
an $A$-graded $R$-bimodule $L$ equipped with an $R$-bilinear, 
$R$-balanced operation $[ \ , \ ]$
for which
\begin{itemize}
\item[(i)]\rule{0ex}{2ex}
$[x,y]\subset L_{|x||y|}$
  \ \ {\em ($A$-grading)},
\item[(ii)]\rule{0ex}{2.5ex}
 $[x,y] = - \epsilon ( |x|,|y|)\ [y,x]$\ \ {\em ($\epsilon$-antisymmetry)}, and,
\item[(iii)]\rule[-1.5ex]{0ex}{4ex}
 $0=\epsilon(|z|,|x|)\, [x, [y,z]] + \epsilon(|x|,|y|)\,
  [y,[z,x]] + \epsilon(|y|,|z|)\, [z, [x,y]]$
   \   {\em ($\epsilon$-Jacobi identity)}
\end{itemize}
for all 
$A$-homogeneous $x,y,z\in L$.
We say $L$ is an {\em ungraded color Lie ring} if it 
only satisfies (ii) and (iii) above.
}\end{definition}

Color Lie rings generalize other Lie algebras:
\begin{itemize}
\item
When $A=1$, a color Lie ring is the usual notion of Lie ring over $R$. 
\item
When $R=\k$ (a field), a color Lie ring is the usual notion of color Lie algebra
or color Lie superalgebra.
\item
When $R=\k$ and $A= \Z/2\Z$, a color Lie ring is a Lie superalgebra. 
We give one such example next. 
\end{itemize}
Note that parts~(i) and~(ii) in the definition a color Lie ring $L$ 
force the group $A$ to be abelian since 
$[L_a, L_b]\subset L_{ab}$ and $[L_b,L_a] \subset L_{ba}$
for all $a,b$ in $A$.

\vspace{3ex}

\begin{ex}\label{superalg}
$L= {\mathfrak{gl}}(1 | 1)$ is the Lie superalgebra defined as follows.
Let $R = \k $ be a field of characteristic not~$2$. 
Let $A = \Z/2\Z = \{ 0,1\}$ and $\epsilon(a,b) = (-1)^{ab}$
for all $a,b$ in $A$.
Let $L$ be the set of $2\times 2$ matrices
$M =(m_{ij})$ over $\k$, 
where for  $M \neq 0$ we define $|M|=0$ if $m_{12}=m_{21}=0$, and 
$|M|=1$ if $m_{11}=m_{22}=0$.
Define $[ M,M']$ to be $MM' - \epsilon(|M|,|M'|) M'M$ if $M,M'$ are $A$-homogeneous; then $L$ is a color Lie ring.  In fact, one can define a color Lie ring for any $A$-graded associative algebra $B$ by replacing the usual commutator with the color bracket $[u,v] = uv - \epsilon(|u|,|v|) vu$ for homogeneous elements $u,v \in B$ (\cite[\S 3, p.~723]{KharchenkoShestakov}).
Note that in $L= {\mathfrak{gl}}(1 | 1)$ there are matrices $M$ for which $[M,M]\neq 0$, for example,
let $M = (m_{ij})$ with $m_{11}=m_{22}=0$ and both $m_{12}\neq 0$ and $m_{21}\neq  0$. 
Also note that if $M = (m_{ij})$ with $m_{11}=m_{12}=m_{22}=0$ and $m_{21}\neq 0$,
then $[ M,M ]=0$, but $\epsilon(|M|,|M|)=-1$,
forcing $M^2=0$ in the universal enveloping algebra $\mathcal{U}(\mathfrak{gl}(1 | 1))$
 (see Definition~\ref{def of env alg}   below), so that $\mathcal{U}(\mathfrak{gl}(1 | 1))$
 has nilpotent elements.
\end{ex}

\vspace{3ex}

We will see later in Section~\ref{sec:cat-th} that color Lie rings are
Lie algebras in symmetric monoidal categories. 

\vspace{3ex}

\subsection*{Color universal enveloping algebras}
There is a notion of universal enveloping algebra for color Lie rings: 
\begin{definition}\label{def of env alg}{\em 
Let $L$ be a color Lie ring (either graded or ungraded) over $R$. 
Consider its tensor algebra $T_R(L)=\bigoplus_{n\geq 0} L^{\otimes_R n}$  and let $J$ be the ideal in $T_R(L)$ defined by generators as 
\[ 
   J := \left(\{ u \otimes_R v-\epsilon(|u| , |v|)\ v\ot_R u -[u,v] : 
    \mbox{ $A$-homogeneous } u,v \mbox{ in } L \} \right)  . \]
The {\em universal enveloping algebra} of $L$ is the 
quotient
\[ \U(L):=T_R(L)/J. \]
}\end{definition}

\subsection*{Yetter-Drinfeld color Lie rings}\label{sec:Yetter}
A special case of color Lie ring arises from groups
acting linearly
on finite dimensional vector spaces.
Suppose $V\cong \k^n$ and $G$ is a finite group acting linearly on $V$.  
We use left superscript to denote the group action in
order to distinguish it from multiplication in a corresponding
skew group algebra, writing $\ ^gv$ for the action of
$g$ in $G$ on $v$ in $V$.

Set $R=\kG$ and view $L=V\ot \kG$  as  an $R$-bimodule via the
natural action
\begin{equation}\label{eqn:YD2}
    g (v\ot h)g' = {}^gv\ot ghg' 
    \quad\text{for all }\ g,g',h\in G, \ v\in V
    .
\end{equation} 
When $G$ is abelian, this bimodule structure
gives rise to a special kind of color Lie ring over $\kG$,
one that is compatible with a Yetter-Drinfeld structure,
as we define next.

\begin{definition}\label{def:YDcolorLie}
{\em
  We say that an $(A,\epsilon)$-color
  Lie ring $L=V\ot\kG$
over $\kG$ (with $\kG$-bimodule structure
above) is {\em Yetter-Drinfeld}
if \begin{itemize}
\item[(i)] 
the $A$-grading on $L$ is induced from $A$-gradings on $G$
and on $V$, i.e., 
$$|v\ot g|=|v||g| \quad
\text{for all } g \in G \text{ and $A$-homogeneous }
v \text{ in } V  , \text{ and}$$
\item[(ii)] ${}^g v = \epsilon ( |g|, |v|)\, v$
for all $g$ in $G$ and $A$-homogeneous $v$ in $V$.
\end{itemize}
We make the same definition
for ungraded color Lie rings.}
\end{definition}

When a color Lie ring $L$ is Yetter-Drinfeld,
we may choose a basis $v_1, \ldots, v_n$ of $V$
which is $A$-homogeneous since
$V$ is $A$-graded as a $\kG$-bimodule.  Condition~(ii)
in the definition then guarantees that the abelian group $G$ acts diagonally
with respect $v_1,\ldots, v_n$.

\subsection*{Yetter-Drinfeld modules}
We mention briefly the connection with Yetter-Drinfeld modules for context.
For $G$ any finite group, a {\em Yetter-Drinfeld $\k G$-module}
is a $\k G$-module $M$ that is also $G$-graded with
$M = \oplus _{g\in G} M_g$ in such a way that
$ {}^g(M_h) = M_{ghg^{-1}}$ for all $g,h$ in $G$.
We may view 
$V\ot \k G$ as a Yetter-Drinfeld module:
It is $G$-graded with $(V\ot \kG) _g = V\ot \k g$
for all $g$ in $G$, 
it is a (left) $\kG$-module via 
\begin{equation}\label{eqn:YD1}
{}^g(v\ot h)=\, ^{g}v\ot ghg^{-1}
\quad\text{for all }\
g,h\in G \text{ and } v\in V\, ,
\end{equation}
and $ {}^g ((V\ot\kG)_h) = (V\ot \kG)_{ghg^{-1}}$ for all $g,h$ in $G$. 
Thus one may consider
the left action~(\ref{eqn:YD1}) to be
an adjoint action of $G$.


\section{Quantum Drinfeld orbifold algebras $\Hqk$}
\label{sec:qdoa} 

In this section, we consider some deformations of skew group 
algebras, prove a PBW theorem needed later,
and give many new examples. 

\subsection*{Quantum systems of parameters} 

We define a {\em quantum system of parameters}
(or {\em quantum scalars} for short)
to be a set of nonzero scalars
${\bf q}:=\{ q_{ij} \}$
with
$q_{ii}=1$  and $q_{ji}=q_{ij}^{-1}$ for all $i, j$ ($1\leq i,j\leq n$).


\subsection*{Quantum symmetric algebras}  
Let $V$ be a $\k$-vector space of dimension $n$,
and fix a basis $v_1,\ldots, v_n$ of $V$.
Let ${\bf{q}}$ be a quantum system of parameters.  
Recall that the {\em quantum symmetric algebra} (or skew polynomial ring)
determined by ${\bf q}$ and $v_1,\ldots, v_n$ 
 is the  $\k$-algebra
$$
S_{\bf q}(V):= \k\left<v_1, \ldots, v_n\right>/(v_i v_j = q_{ij} v_j v_i
\  \text{ for } 1\leq i, j\leq n) \,  .
$$
Note that $S_{\bf q}(V)$ has the {\em PBW property}:
As a $\k$-vector space, $S_{\bf q}(V)$ has
basis 
\[
   \{ v_1^{m_1}v_2^{m_2}\cdots v_n^{m_n}: m_i\in \N \}.
\]

\subsection*{Skew group algebras} 
Let $S$ be a $\k$-algebra on which a finite group $G$ acts by automorphisms.
The {\em skew group algebra} (or semidirect product) $S\rtimes G$ is the $\k$-algebra generated by $S$ and $G$ with multiplication given by
\[
    (r\ g) \cdot (s\ h) := r ({}^gs) \ gh
\]
for all $r,s\in S$ and $g,h\in G$. 
(As a $\k$-vector space, $S\rtimes G$ is isomorphic to $S\ot \kG$.)

We will use this construction particularly in the following setting. 
Let $G$ be a finite group, and let $V$ be a $\kG$-module.
Assume the linear action of $G$ on $V$ extends
to an action by graded automorphisms on $S_{\bf q}(V)$
(letting $\deg v=1$ for all $v\in V$).  In other words,
assume that in $S_{\bf q}(V)$, 
$$
\ ^gv_i\, ^gv_j = q_{ij} \ ^g v_j\, ^gv_i
\quad\text{ for all } i,j \ \text{ and } g \in G\ .
$$
Groups acting diagonally on the choice $v_1,\ldots,v_n$ of basis 
of $V$ always extend to an action on $S_{\bf q}(V)$,
but many other group actions do {\em not} extend.
For our main results, 
we will restrict to diagonal actions.

\subsection*{Parameter functions} 
Let $T(V) = T_{\k}(V)$ be the free algebra over $\k$ 
generated by $V$ (i.e., the tensor algebra 
with tensor symbols suppressed).  The group $G$ acts on $T(V)$ 
by automorphisms (extending its action on $V$). 
Consider a quotient of the skew group algebra  $T(V)\rtimes G$
by relations that lower the degree of the commutators
$v_iv_j - q_{ij}v_jv_i$,
viewing $\kG$ in degree zero.
Specifically, consider  a {\em parameter function} 
\begin{equation}\label{eqn:parfun}
  \kappa: V\times V \rightarrow V\otimes \kG 
\end{equation}
that is bilinear
and {\em quantum antisymmetric}, i.e.,
\[
   \kappa(v_i, v_j)=-q_{ij}\ \kappa(v_j, v_i) \, ,
\]
that records such relations.
We decompose $\kappa$ when convenient, writing
$$
\kappa(v_i, v_j)
=\sum_{g\in G} \kappa_g(v_i, v_j) g
$$
with each $\kappa_g:V\times V \rightarrow V$.
We note that imposing bilinearity is only for convenience; 
values of $\kappa$ on the given basis 
determine the algebra $\mathcal{H}_{{\bf q},\kappa}$
defined by~(\ref{eqn:Hqk}) below.

\subsection*{Quantum Drinfeld orbifold algebras} 

Given a quantum system of parameters ${\bf{q}}$ and a parameter function $\kappa$
as in~(\ref{eqn:parfun}), let
\begin{equation}\label{eqn:Hqk}
\mathcal{H}_{\bf q, \kappa}:=
(T(V)\rtimes G)/\big(v_i v_j - q_{ij} v_j v_i-\kappa(v_i,v_j):
1\leq i,j\leq n\big).
\end{equation}
We say that $\mathcal{H}_{\bf q,\kappa}$
satisfies the {\em PBW property} 
if the image of the set  
$$
\{v_1^{m_1}v_2^{m_2}\cdots v_n^{m_n}g:
m_i\in \N, g\in G\} \, , 
$$
under the quotient map $T(V)\rtimes G\rightarrow \Hqk$, 
is a basis for $\mathcal{H}_{\bf q,\kappa}$
as a $\k$-vector space.
In this case, we say  
 $\mathcal{H}_{\bf q, \kappa}$
is a {\em PBW deformation} of $S_{\bf q}(V)\rtimes G$,
and call $\mathcal{H}_{\bf q, \kappa}$
a  {\em quantum Drinfeld orbifold algebra}. 
We identify every $v\ot g$ in $V\ot\kG$ with its image
$vg$ in $\mathcal{H}_{{\bf q},\kappa}$.

\subsection*{PBW conditions}
We derive necessary and sufficient conditions
for $\cH_{\bf q, \kappa}$ to satisfy the PBW property in terms of properties of $\kappa$.
Related PBW theorems appear in
Shepler and Levandovskyy~\cite{LevandovskyyShepler},
Shepler and Witherspoon~\cite{DOA}, 
and Shroff~\cite{Shroff}.
These previous theorems stated results 
in a way that allowed direct comparison with homological information,
but here we obtain a PBW result in Corollary~\ref{PBWconditions'}
of a different flavor
in order to connect with color Lie rings.

In the theorem below,
we extend
the parameter function $\kappa$ from
domain $V\times V$ to domain $(V\ot \kG)\times (V\ot\kG)$ by setting
$$
   \kappa (vg , wh) = \kappa(v, {}^gw) gh
\quad\text{ for all }g,h\in G\ \text{and }\  v,w\in V .
$$
We say that $\kappa$ is {\em $G$-invariant} in Condition~(1)
of the theorem below if it is invariant
with respect to the {\em adjoint action} of $G$ on the bimodule $V\ot\kG$
(as in Equation~(\ref{eqn:YD1})):
\begin{equation}
  \label{adjointaction}
        {}^g(v\ot h) = \ {}^gv\ot ghg^{-1} =\ {}^g v\ot h
\quad\text{ for all } g,h\in G
\text{ and } v\in V.
\end{equation}
We record the action of $G$ acting diagonally on a fixed basis $v_1,\ldots, v_n$
of $V$ with linear characters $\chi_i:G\rightarrow \k^*$ giving
the $i$-th diagonal entries:
\[
  {}^g v_i =\chi_i(g) \ v_i
\quad  \text{ for } 1\leq i \leq n.
\]
 
\begin{thm}\label{PBWconditions}
Let $G$ be an abelian group acting diagonally 
on $V$ with respect to a $\k$-basis $v_1, \ldots, v_n$.
Let $\kappa: V\times V\rightarrow V\ot \k G$ be a parameter function
defining the algebra
$\Hqk$ as in~(\ref{eqn:Hqk}), and write 
$\kappa(v_i, v_j)=\sum_{1\leq r\leq n,\, g\in G} c_r^{ijg}\ v_r\ot g$ in $V\ot\kG$
for some $c^{ijg}_r\in \k$. 
Then 
$\mathcal{H}_{\bf q, \kappa}$ satisfies the PBW property if
and only if 
\begin{itemize}
\item[(1)]
the function $\kappa$ is $G$-invariant,
\item[(2)]
for all $g$ in $G$ and all distinct $1\leq i,j,k\leq n$,
  $$
\begin{aligned}
  0&=\big(\chi_k(g)-q_{jk}\, q_{ik}\, q_{kr}\big)c_r^{ijg}
 \qquad \text{for any } r \neq i,j\, , \\
0&=\big(1-q_{ij}\chi_i(g)\big)\, c_k^{\, jkg}+\big(q_{jk}-\chi_k(g)\big)\, c_i^{ijg} \, ,
\qquad \text{and} 
\end{aligned}
$$
\item[(3)]
for all $g$ in $G$ and all distinct $1\leq i,j,k\leq n$, 
\rule[0ex]{0ex}{3ex}
\begin{equation*}
\begin{aligned}
  \quad\quad 0 =\
  \sum_{\circlearrowleft}
   \big(\chi_k(g)-q_{jk}\big)\, c_i^{ijg} \kappa(v_k, v_i)
\ + \
\sum_{\circlearrowleft}
q_{jk}\ \kappa\big(v_k,\kappa_g(v_i,v_j)\big) . 
\end{aligned}
\end{equation*}
\end{itemize}
\end{thm}
\begin{proof}
We use the Diamond-Composition Lemma as outlined 
in~\cite{LevandovskyyShepler}, 
~\cite{DOA},
and~\cite{Shroff} to write elements of $\Hqk$ in PBW form, up to cyclic permutation.
We consider ``overlaps'' arising from
$0={}^h (v_i v_j) - ({}^h v_i)( {}^hv_j)$ to obtain
the condition that
$ \kappa_g ( {}^hv_i , {}^h v_j)
=\ {}^h\big(\kappa_g (v_i,v_j)\big) $ 
for all $i,j$ and all $g,h$ in $G$,
yielding the first condition of the theorem.
(See~\cite[Prop.~9.3]{LevandovskyyShepler} 
for a substitute for $G$-invariance
when $G$ does not act diagonally.)

We consider overlaps arising from
$(v_k v_j) v_i = v_k (v_j v_i)$
to obtain the second and third conditions.
We use the relations of the algebra
to write 
$0=(v_k v_j) v_i - v_k (v_j v_i)$
in the span of the PBW basis
$\{v_1^{m_1}\cdots v_n^{m_n}g: m_i\in\NN, g\in G\}$.
We put $G$ in degree~$0$
and observe that straightforward calculations
show that
terms of degree~3 
in $v_1, \ldots, v_n$ cancel.
Since the group acts diagonally, 
direct calculations give the 
coefficient of a fixed $g$ in $G$
(as an element of $T(V)\ot \kG$) as
\begin{equation*}
\begin{aligned}
\sum_{\circlearrowleft} 
q_{jk}\ v_k\ \kappa_g(v_i, v_j) 
-q_{ki}\ \chi_k(g)\ \kappa_g(v_i, v_j)\ v_k\, . \\
\end{aligned}
\end{equation*}
We assume $i,j,k$ are distinct as this expression
trivially vanishes otherwise. We expand each 
$\kappa_g(v_i, v_j)$ (for fixed $g$) as
$\sum_{1\leq r\leq n} c_r^{ij} v_r$
for $c^{ij}_r=c^{ijg}_r\in \k$, obtaining
\begin{equation*}
\begin{aligned}
0&=
\sum_{\circlearrowleft} 
\sum_{1\leq r\leq n}
q_{jk}\ c_r^{ij}\ v_k v_r
-q_{ki}\ \chi_k(g)\ c_r^{ij}\ v_r v_k\, .
\end{aligned}
\end{equation*}
Next, we exchange the order of $v_r$ and $v_k$
so that indices increase, obtaining a sum
over cyclic permutations of $i,j,k$ of
\begin{equation*}
\begin{aligned}
\sum_{k\leq r} &
q_{jk}\  c_r^{ij}\ v_k v_r
-q_{ki}\ \chi_k(g)\ c_r^{ij}\
\big(q_{rk}\ v_k v_r + \kappa(v_r, v_k)\big)\\
&\quad\quad\quad
+ \sum_{k>r}
q_{jk}\ c_r^{ij}\ \big(q_{kr}\ v_r v_k + \kappa(v_k, v_r)\big)
-q_{ki}\ \chi_k(g)\ c_r^{ij}\ v_r v_k\ .
\end{aligned}
\end{equation*}
The terms of degree 2 vanish exactly when
\begin{equation}
  \label{deg2}
\begin{aligned}
0=
\sum_{\circlearrowleft} 
\Big(
\sum_{k\leq r}
q_{rk}\, \big(q_{jk}\, q_{kr}  
-q_{ki}\, \chi_k(g)
\big)\, c_r^{ij}\ v_k v_r  
+
\sum_{k>r}
\big(q_{jk} \,
q_{kr} -q_{ki}\, \chi_k(g)\big)\, c_r^{ij}\ v_r v_k\Big) \ .
\end{aligned}
\end{equation}
We combine like terms and see
that~\eqref{deg2} holds exactly when six equations hold,
obtained by taking cyclic permutations of indices on equations
\begin{gather}
\big(q_{jk}-q_{ki}\, q_{rk}\, \chi_k(g)\big)\, c_r^{ij}=0 
\qquad \text{for } r \neq i,j \label{Con21}\, , \\
\big(1-q_{ij}\, \chi_i(g)\big)\, c_k^{jk}+\big(q_{jk}-\chi_k(g)\big)\,c_i^{ij}=0 \label{Con22}\, ,
\end{gather}
whence Condition~(2) follows.

We collect terms of degree $1$; these vanish exactly when
\begin{equation*}\label{deg1terms}
\begin{aligned}
  0& =\sum_{\circlearrowleft} \Big( 
  \sum_{k\leq r}  -q_{ki}\ \chi_k(g)\ c_r^{ij} 
\ \kappa(v_r, v_k)
+ \sum_{k>r} q_{jk}\ c_r^{ij}\ \kappa(v_k, v_r)\Big)
\\
&=
\sum_{\circlearrowleft} \Big(
\sum_{k\leq r} q_{ki}\ q_{rk}\ \chi_k(g)\ 
c_r^{ij}\  \kappa(v_k, v_r)
+ \sum_{k>r} q_{jk}\ c_r^{ij}\ \kappa(v_k, v_r)\Big).
\end{aligned}
\end{equation*}
We add and subtract the second sum
over $k>r$ but taken over
$k\leq r$ instead to obtain
\begin{equation*}\label{onehand}
  \begin{aligned}
    0&=\sum_{\circlearrowleft} 
  \sum_{k\leq r} \big(q_{ki} \, q_{rk}\, \chi_k(g)-q_{jk} \big)\ c_r^{ij}  
  \ \kappa(v_k, v_r)
  +   \sum_{\circlearrowleft}
  \sum_{r} q_{jk}\ c_r^{ij}\ \kappa(v_k, v_r)\, .
  \end{aligned}
\end{equation*}
The second sum is just
$\sum_{\circlearrowleft}\,
q_{jk}\, \kappa\big(v_k,\kappa_g(v_i,v_j)\big)$.
Equation~\eqref{Con21} implies that terms
in the first sum with $r\neq i,j$ vanish;
we rewrite the remaining terms using
Equation~\eqref{Con22} and the fact that $\kappa$
is quantum antisymmetric and
obtain
$\sum_{\circlearrowleft}
\big(\chi_{k}(g)-q_{jk}\big) c_i^{ij}$.
Condition~(3) follows.
\end{proof}

\vspace{3ex}

\begin{rem}
The conditions in Theorem~\ref{PBWconditions}
  can be written alternatively as 
\begin{itemize}
\item[(1)]
the function $\kappa$ is $G$-invariant,
\item[(2)]
$
0=\sum_{\circlearrowleft} 
 q_{jk}\, v_k\, \kappa_g(v_i, v_j) 
-q_{ki}\, \chi_k(g)\, \kappa_g(v_i, v_j)\, v_k
\text{ in }\ S_{\bf q}(V)\, ,\ \text{ and }
$
\rule[0ex]{0ex}{2.5ex}
\item[(3)]
\rule[0ex]{0ex}{2.5ex}
$0 = \sum_{\circlearrowleft}
q_{ki}\ \chi_k(g)\ \kappa(\kappa(v_i,v_j), v_k)
-
q_{jk}\ \kappa(v_k,\kappa(v_i,v_j))
$\quad in $V\ot \kG \, $,
\end{itemize}
for all $g$ in $G$ and
all distinct $i,j,k$.
\end{rem}

\vspace{3ex}

We will see in the proof of Theorem~\ref{thm:main1} that
the color Lie Jacobi identity is equivalent to the
condition $0=\sum_{\circlearrowleft}
q_{jk}\ \kappa\big(v_k,\kappa_g(v_i,v_j)\big)$
for distinct $i,j,k$.
Hence, we are interested in a condition
that recovers this expression
from Theorem~\ref{PBWconditions}(3), i.e.,
a condition that forces half of that third PBW condition
to vanish.
This condition, defined next, is a precursor to a stronger condition that
we will need later to guarantee a Hopf structure on the quantum
Drinfeld orbifold algebra $\Hqk$. 

\vspace{2ex}

\begin{definition}\label{vanishingcondition}
  \em{A
    quotient algebra $\Hqk$ as in~(\ref{eqn:Hqk}) satisfies
    the {\em vanishing condition} if for all $g$ in $G$,
    distinct $1\leq i,j,k \leq n$,
  and $1\leq r \leq n$,
$$ \,^gv_k = q_{ik}\, q_{jk}\, q_{kr}\ v_k$$
  whenever the coefficient of $v_r\ g$ is nonzero
  in $\kappa(v_i, v_j)\in V\ot \kG$.}
\end{definition}

\vspace{2ex}

This vanishing condition indeed implies a simplified PBW theorem
as a corollary of Theorem~\ref{PBWconditions}:

\begin{cor}\label{PBWconditions-funnycondition}
Let $G$ be an abelian group acting diagonally 
on $V$ with respect to a $\k$-basis $v_1, \ldots, v_n$.
Let $\kappa: V\times V\rightarrow V\ot \k G$ be a parameter function
defining the algebra $\Hqk$ as in~(\ref{eqn:Hqk}) and 
satisfying vanishing condition~(\ref{vanishingcondition}). Then 
$\mathcal{H}_{\bf q, \kappa}$ satisfies the PBW property if
and only if
\begin{itemize}
\item
  $\kappa$ is $G$-invariant,
  and
\item
\rule[0ex]{0ex}{2.5ex}
$0=\sum_{\circlearrowleft} 
q_{jk}\ \kappa\big(v_k,\kappa(v_i,v_j)\big)
\quad\text{ for distinct }\ \
1\leq i,j,k \leq n \, $.
\end{itemize}
\end{cor}

\vspace{2ex}

Groups acting on $V$ without fixed points
will provide us with a wide class of examples where
the vanishing condition holds and thus a color Jacobi identity holds. 
To establish this connection,
we first need a lemma.

\begin{lemma}\label{Fixed point}
Let $G$ and $\Hqk$ be as in Theorem~\ref{PBWconditions}.
\begin{enumerate}[(i)]
\item If $\kappa$ is $G$-invariant, then
$\chi_i\chi_j=\chi_r$ for $1\leq i,j,r \leq n$
whenever
$v_r\, g$ has a nonzero coefficient in $\kappa(v_i,v_j)$.
\item If the action of $G$ is fixed point free,
  then Condition~(2) in Theorem~\ref{PBWconditions} is equivalent
  to the vanishing condition~(\ref{vanishingcondition}).
\end{enumerate}
\end{lemma}
\begin{proof}
If $\kappa$ is $G$-invariant, then for all $g,h$ in $G$,
$$
\begin{aligned}
\sum_r \chi_r(h)\, c_r^{ijg}\, v_r
& \ =\ {}^h\big(\sum_r c_r^{ijg}\, v_r) 
\ =\ {}^h\big(\kappa_g(v_i, v_j)\big)\\
& \ =\ \kappa_g({}^h v_i, {}^h v_j)
\ =\  \chi_i(h)\, \chi_j(h)\ \kappa_g(v_i, v_j)
\ =\ \sum_r
\chi_i(h)\chi_j(h)\, c_r^{ijg} \, v_r\, .
\end{aligned}
$$
This implies that
$\chi_i(h)\chi_j(h)=\chi_r(h)$ for all $h$ in $G$
whenever some $c_r^{ijg}\neq 0$.
In particular, in the case $r=j$,
we have $\chi_i(h)=1$ for all $h \in G$ whenever $c_r^{ijg}\neq 0$,
i.e., $G$ fixes $v_i$.
Thus, if the action of $G$ is fixed point free,
then $c_i^{ijg} =0=c_j^{ijg}$ for all distinct $i,j$ and all $g$ in $G$,
and the second part of Condition~(2) is trivial.
\end{proof}

Corollary~\ref{PBWconditions-funnycondition} and
  Lemma~\ref{Fixed point} together imply the following
  PBW result.

\begin{cor}\label{PBWconditions'}
The three PBW conditions of
Theorem~\ref{PBWconditions} 
are implied by conditions
\begin{itemize}
\item[(1$'$)]
the function $\kappa$ is $G$-invariant,
\item[(2$'$)] vanishing condition~(\ref{vanishingcondition}) holds, and
\rule[-1.5ex]{0ex}{4ex}
\item[(3$'$)]
  \rule[-1.5ex]{0ex}{3ex}
$0=\sum_{\circlearrowleft}
q_{jk}\ \kappa\big(v_k,\kappa(v_i,v_j)\big)$
  for all distinct $1\leq i,j,k \leq n$.
\end{itemize}
In addition, if the action of $G$ on $V$ is fixed point free,
then the three conditions in 
Theorem~\ref{PBWconditions} may be replaced by the three
above conditions.
\end{cor}

\vspace{3ex}

There are many examples in the literature of quotient algebras defined 
by~(\ref{eqn:Hqk})
but with $\kappa$ taking values in $\k G$ instead of
in $V\ot \k G$, particularly 
in the case where $q_{ij}=1$ for all $i,j$. 
This includes 
the noncommutative deformations of Kleinian singularities \cite{CBH}
in which $n=2$ and $G < {\rm {SL}}_2(\k)$. 
There are also examples for more general quantum systems 
of parameters ${\bf q}$, however for those
in the literature, again $\kappa$ typically takes values in $\kG$. 
We give here some examples of a different nature,  focusing in this paper 
on the less studied case where $\kappa$ takes values
in $V\ot \kG$.
For all these examples,
it will follow from results in the next section
that $L=V\ot \kG$ is an ungraded Yetter-Drinfeld color Lie ring.
One may check that they all satisfy the vanishing condition~(\ref{vanishingcondition}).
Examples~\ref{ex:2} and~\ref{ex:3} in fact satisfy a stronger
condition, the strong vanishing condition~(\ref{strongvanishingcondition}) that will
be defined in Section~\ref{gradedcolorLierings}.

\vspace{1ex}

\begin{ex}\label{ex:1}
  Let $V=\k^3$ with basis $v_1,v_2,v_3$.
  Let $q$ in $\k$ be a primitive $m$-th root of unity
  for some fixed $m\geq 2$.
  Let $G\cong \Z/m\Z\times \Z/m\Z$
  be generated by diagonal matrices
$
g_1=\text{diag}(q,1,q)$ and 
$g_2=\text{diag}(1,q,q)$ acting on $V$.
The action of $G$ on $V$ induces 
an action on $S_{\bf q}(V)$ where $q_{12}=q_{13}=q$ and
$q_{23}=1$. Then the algebra
\[
    \mathcal{H}=(T(V)\rtimes G)/ (v_1v_2 - q v_2v_1 - qv_3\, g_1 , \ 
   v_1v_3 - qv_3v_1, \ v_2v_3 - v_3v_2 )
\]
satisfies the PBW property by 
Theorem~\ref{PBWconditions}.
Corollary~\ref{cor:main1fixedpointfree}
will imply that
$\mathcal{H}$ is the universal enveloping algebra
$\mathcal{U}(L)$ of 
an ungraded color Lie ring  
$L=V\ot \kG$ with 
\[
   [v_1,v_2] = q\, v_3\, g_1, \ \ \ [v_1,v_3]=0, \ \ \ [v_2,v_3]=0 .
\]
See~\cite[Proposition~3.5, Remark~3.6, and Remark~3.11]{Shakalli},
where this example appears.
\end{ex}

\vspace{1ex}

\begin{ex}\label{ex:2}
Let $V=\k^3$ with basis $v_1,v_2,v_3$.
Let $G=\Z/2\Z$ act on $V$ by the diagonal matrix $g=\text{diag}(-1,-1,1)$.
Fix some nonzero $q$ in $\k$. 
The action of $G$ on $V$ induces an action on $S_{{\bf q}}(V)$ for $q_{12}=q^{-1}$,
$q_{13}=-q^{-1}$, and $q_{23}=-q$. 
Fix some $\lambda$ in $\k$.
Then
\[
\mathcal{H}_\lambda=(T(V)\rtimes G) / (v_1v_2-q^{-1}v_2v_1-\lambda v_3\, g , \
   v_1v_3+q^{-1}v_3v_1, \ v_2v_3 + q v_3v_2 )
\]
satisfies the PBW property by
Theorem~\ref{PBWconditions}. 
Theorem~\ref{thm:main1} will imply that
$\mathcal{H}_\lambda \cong \mathcal{U}(L_\lambda)$
for
a 1-parameter family of graded color Lie rings $L_\lambda=V\ot\kG$ with
brackets
\[
  [v_1,v_2]=\lambda v_3\, g , \ \ \ [v_1,v_3]= 0 , \ \ \ [v_2,v_3]=0 .
\]
\end{ex}

\vspace{1ex}

\begin{ex}\label{ex:3}
  Let $V=\k^3$ with basis $v_1,v_2,v_3$.
  Let $q$ in $\k$ be a primitive $m$-th root of unity for some
  fixed $m\geq 2$ and let $G=\Z/m\Z$
  act on $V$ by generator $g=\diag(q,1,1)$.
  Let $\lambda,p$ in $\k$ be nonzero. 
The action of $G$ on $V$ induces an action on $S_{{\bf q}}(V)$ where
$q_{12}=q^{-1}$, $q_{13}=p$, and $q_{23}=1$.
Then
\[
\mathcal{H}_{\lambda,p}
=( T(V)\rtimes G)/ (v_1v_2-q^{-1}v_2v_1-\lambda v_1\, g , \
   v_1v_3-pv_3v_1, \ v_2v_3-v_3v_2 )
\]
satisfies the PBW property by Theorem~\ref{PBWconditions}.
Theorem~\ref{thm:main1} will imply that
$\mathcal{H}_{\lambda,p} \cong \mathcal{U}(L_{\lambda,p})$
for
a $2$-parameter family of graded
color Lie rings $L_{\lambda,p} = V\ot \kG$ with
\[
   [v_1,v_2]= \lambda\, v_1\, g , \ \ \ [v_1,v_3]=0 , \ \ \ [v_2,v_3]=0 .
\]
\end{ex}

\vspace{1ex}

\begin{ex}\label{ex:4}
  Fix some $n\geq 1$.
Let $V=\k^{2n}$ with basis $v_1,\ldots, v_{2n}$.
Consider $G=(\Z/2\Z)^{n}$ acting on $V$ by generators $g_1,\ldots, g_{n}$
with ${}^{g_i} v_j = (-1)^{\delta_{i,j}} v_j$
for $1\leq i\leq n$, $1\leq j\leq 2n$.
Set $q_{ij} = (-1)^{\delta_{j-i,n}}$ for $1\leq i<j\leq 2n$.
The action of $G$ on $V$ induces an action on $S_{{\bf q}}(V)$.
Let $\lambda_i$ be a scalar in $\k$ ($1\leq i\leq n$).
Then
\[
\begin{aligned}
  \mathcal{H}_{\boldsymbol{\lambda}}
  =(T(V)\rtimes G)/(&v_iv_{n+i} + v_{n+i}v_i -\lambda_i v_i\, g_i\quad \mbox{ for } 
   1\leq i\leq n,\\
  & v_i v_j -  v_j v_i \quad \mbox{ for } 1\leq i<j\leq 2n \mbox{ with } j\neq n+i )
\end{aligned}
\]
satisfies the PBW property (by 
Theorem~\ref{PBWconditions})
for $\boldsymbol{\lambda}=(\lambda_1, \ldots, \lambda_n)$,
$L_{\boldsymbol{\lambda}} = V\ot \kG$ is an ungraded color Lie ring 
with
$$
\begin{aligned}
  {[}v_i, v_{n+i}{]} & = & \lambda_i\, v_i\, g_i \
  &\quad\mbox{ for } 1\leq i\leq n , & & \\
  {[}v_i,v_j{]} & = & 0 \
  &\quad\mbox{ for } 1\leq i<j\leq 2n \mbox{ with } j\neq n+i , & & 
\end{aligned}
$$
and Theorem~\ref{thm:main1} will imply that $\mathcal{H}_{\boldsymbol{\lambda}}
\cong \mathcal{U}(L_{\boldsymbol{\lambda}})$.
\end{ex}

\section{Drinfeld orbifold algebras
  defining ungraded color Lie rings}
\label{FirstMainThm}
We now highlight the distinction between the graded
and ungraded color Lie rings.
We show in this section
that every quantum Drinfeld orbifold algebra
with $G$ acting fixed point free on $V$
is the universal enveloping algebra of
some ungraded Yetter-Drinfeld color Lie ring.

As before, let $G$ be a finite abelian group acting diagonally
on $V=\k^n$ with basis $v_1,\ldots, v_n$
and use linear characters
$   \chi_i: G\rightarrow \k^* $ 
to record the $i$-th diagonal entries of each $g$ in $G$
by setting $  {}^g v_i = \chi_i(g) v_i $ for $1\leq i\leq n$.

\subsection*{Natural grading used to construct  Yetter-Drinfeld color Lie rings}
\label{naturalgrading}
We work with a choice of grading group $A$ throughout this section: 
Set $A=\ZZ^n\times G$.
The space 
$$L=V\ot \kG = \bigoplus_{i,\, g} \k (v_i\ot g)$$
is naturally graded as a vector space
by the abelian group $A$ after
setting
\begin{equation*}
|v_i|=(a_i , 1_G),
\quad
|g|=(0,g),
\quad\text{and}\quad
|v_i \ot g| = |v_i| |g|
\quad\text{ for all } g\in G
\end{equation*}
where
$a_1, \ldots, a_n$
is the standard basis of  $\ZZ^n$.
We will use this specific grading
throughout this section
to relate quantum Drinfeld orbifold algebras
to color Lie rings.

\subsection*{Quantum Drinfeld orbifold algebras
  as universal enveloping algebras}
We show that quantum Drinfeld orbifold algebras
 define ungraded color Lie rings
 when the underlying action of $G$ is fixed point free.
 We begin with a more general theorem that merely requires the
 vanishing condition~(\ref{vanishingcondition}) on quotient algebras $\Hqk$
defined as in
 ~(\ref{eqn:Hqk}). 
 
 \begin{thm}
   \label{thm:main1}
  Let $G$ be a finite abelian group acting diagonally
on $V=\k^n$ with respect to a basis $v_1,\ldots,v_n$.
Suppose that $\Hqk$ 
is 
a quantum Drinfeld orbifold algebra
for some parameter function $\kappa: V\times V\rightarrow V\ot \k G$
satisfying vanishing condition~(\ref{vanishingcondition}).
Let $A = \ZZ^n\times G$ and 
consider the $A$-graded vector space $L=V\ot \kG$ as above. Then
\begin{itemize}
\item[(a)]
  There is an
  antisymmetric bicharacter 
  $\epsilon: A\times A \rightarrow \k^*$ with
\begin{equation*}
  \hphantom{xxx}
  \epsilon\big(|v_i \ot g|, |v_j \ot h|\big) 
= q_{ij}\ \chi_i^{-1}(h)\ \chi_j(g)\ 
\ \ \text{for }
g,h\in G,\ 1\leq i,j\leq n\, .
\end{equation*}
\item[(b)]
The space  $L$ is an ungraded
Yetter-Drinfeld  $(A,\epsilon)$-color Lie ring
for a color Lie bracket $[\ ,\ ]: L \times L \rightarrow L$
defined by
\begin{equation*}
   [v_i\ot g,v_j\ot h] = \kappa(v_i, \, ^g v_j)\, g h
   \quad\text{ for }
g,h\in G\, .
\end{equation*}
\item[(c)]
 The algebra
$\cH_{{\bf q},\kappa}$ is isomorphic to the universal enveloping algebra of $L$.
\end{itemize}
\end{thm}
\begin{proof}
  We use the natural $\kG$-bimodule structure
  on  $L = V \otimes \kG$ given by Equation~(\ref{eqn:YD2}).
The function $\epsilon$ given in the statement of the proposition extends to a
bicharacter on $A$ by setting its values on $A\times A$ to be
those determined by the bicharacter condition
and its values on generators. 
Note that this forces
$$\epsilon(|g|,|v_j|)
=\epsilon(|v_i\ot g|\, |v_i\ot 1_G|^{-1},|v_j|)
=\epsilon(|v_i\ot g|,|v_j|)
\ \epsilon\big(|v_i\ot 1_G|,|v_j|\big)^{-1}
=\chi_j(g)$$
and
$\epsilon(|g|,|h|) = 1$
for all $g,h\in G$ and $1\leq i,j\leq n$.
Since $q_{ji}=q_{ij}^{-1}$, the function
$\epsilon$ is antisymmetric.
Then $ ^gv_i=\chi_i(g) v_i = \epsilon(|g|, |v_i|)\, v_i$
and $L$ satisfies the conditions to be Yetter-Drinfeld.
Thus we need only show that $L$ is an ungraded $(A,\epsilon)$-color Lie ring.

\subsubsection*{Bilinear and balanced bracket.}
The bracket $[\ , \ ]$ is $\k G$-bilinear by construction.
We argue that it is also $\kG$-bilinear
with respect to the bimodule structure given by Equation~(\ref{eqn:YD2}).
We use the $G$-invariance of $\kappa$, guaranteed
by Corollary~\ref{PBWconditions-funnycondition},
under the adjoint action of $G$ given in Equation~(\ref{adjointaction}):
For all $a,b,g,h$ in $G$ and $1\leq i, j\leq n$,
$$
\begin{aligned}
\ [\, a(v_i\ot g),\, (v_j\ot h)b\, ]
&=& &[\, ^av_i\ot ag,\ v_j\ot hb\, ]&
&=&
&\kappa(\, ^av_i,\, ^{ag}v_j)\ aghb& \\
&=&
&^a\big(\kappa(v_i,\, ^{g}v_j)\big)\ aghb&
&=&
&a\ \big(\kappa(v_i,\, ^{g}v_j)gh \big)\ b& \\
&=&
&a\ [v_i\ot g,\, v_j\ot h]\ b.&
\end{aligned}
$$

We suppress
tensor symbols on elements of $L$
in the rest of the proof for clarity of notation.

\subsubsection*{Color antisymmetry of bracket}
We check directly that for all $i,j$ and $g,h\in G$, 
$$
\begin{aligned}
\ [v_ig,v_jh]  &\ =\ \kappa( v_i, {}^gv_j)\ gh 
\ =\  \chi_j(g)\ \kappa(v_i,v_j )\ gh
\\
  &\ =\  - q_{ij}\ \chi_j(g)\ \kappa(v_j,v_i)\ gh
   \ =\  - q_{ij}\ \chi_j(g) \chi_i^{-1}(h)\ \kappa (v_jh,v_ig)\\
  & \ =\  - \epsilon( |v_ig| , |v_jh| )\ [v_j h , v_ig] \, . 
\end{aligned}
$$
\subsubsection*{Color Jacobi identity}
For basis vectors $v_i,v_j,v_k$ and $g_i, g_j, g_k$ in $G$, 
$$
\begin{aligned}
\sum_{\circlearrowleft} 
\epsilon( |v_k g_k|,|v_i g_i|) 
&
\big[v_i g_i, [v_j g_j, v_kg_k]\big]\\
&=
\sum_{\circlearrowleft}  
q_{ki} \ \chi_k^{-1}(g_i) \chi_i(g_k)\ 
\kappa\big(v_i, {}^{g_i}\kappa(v_j, {}^{g_j} v_k)\big)
\ g_i g_j g_k \, .
\end{aligned}
$$
By Corollary~\ref{PBWconditions-funnycondition},
$\kappa$ is $G$-invariant, and so this is 
$$
\begin{aligned}
\sum_{\circlearrowleft}  
q_{ki} & \ \chi_k^{-1}(g_i) \chi_i(g_k)\ 
\kappa\big(v_i, \kappa({}^{g_i}v_j, {}^{g_i g_j} v_k)\big)
\ g_i g_j g_k \\
&=
\sum_{\circlearrowleft}  
q_{ki} \ \chi_k^{-1}(g_i) \chi_i(g_k)\ 
\chi_j(g_i) \chi_k(g_i) \chi_k(g_j)
\ \kappa\big(v_i,\kappa(v_j, v_k)\big)\ g_i g_j g_k\\
&=
\chi_i(g_k)\chi_j(g_i) \chi_k(g_j)
\Big(
\sum_{\circlearrowleft}  
q_{ki} \ 
\ \kappa\big(v_i,\kappa(v_j, v_k)\big)\Big)\ g_i g_j g_k
\, ,
\end{aligned}
$$
which vanishes by
Corollary~\ref{PBWconditions-funnycondition} for distinct $i, j, k$.
When $i,j,k$ are not distinct,
then the $\epsilon$-Jacobi identity holds automatically
since $[v,v]=\kappa(v,v)=0$ for all $v$ in $V$ as $\kappa$ is
quantum antisymmetric with $q_{mm}=1$ for all $m$.

\vspace{2ex}

\subsubsection*{Color universal enveloping algebra}
We now check that $\cH_{\bf q,\kappa}$ is isomorphic to $\U(L)$, the 
universal enveloping algebra of $L$. Recall Definition~\ref{def of env alg}: 
$ \ \U(L)=T_{\kG}(V \tensor \kG)/ J $, 
where $J$ is the ideal generated by
\[  \{ v_i g \tensor_{\kG}
v_j h - \epsilon( |v_i g| , |v_jh|) 
( v_j h\tensor_{\kG}  v_i g)-[v_i  g, v_j  h] : 1\leq i,j\leq n, g,h \in G\}. \]
The $\k$-vector space isomorphism
$$
\begin{aligned}
(V \tensor \kG) \tensor_{\kG} (V \tensor \kG)\ &\overset{\cong}{\longrightarrow}
\ V \tensor V \tensor \kG,\\
v_i g \tensor_{\kG} v_j h\ & \mapsto \ v_i \tensor {^g v_j} \tensor gh,
\end{aligned}
$$
induces an isomorphism of $\k$-algebras
$$
\begin{aligned}
  T_{\kG}(V \tensor \kG) \ &\overset{\cong}{\longrightarrow}\ T(V) \rtimes G\, ,\\
  v_i g \tensor_{\kG} v_j h\ &\mapsto\ v_i \ {^g v_j} \ gh.
\end{aligned}
$$
The images of 
generators of the ideal $J$ under this isomorphism
vanish in $\cH_{\bf q, \kappa}$:
\begin{align*}
  v_i g \tensor_{\kG}  v_j h\ -\ &
   \epsilon(|v_ig|,|v_jh|)
( v_j h \ot_{\kG} v_ig) - [v_i g, v_j h]\\
\mapsto\ \ &
v_i \,  {}^g v_j\ gh -\epsilon(|v_ig|,  |v_jh|)\
v_j\,  
{^{h} v_i}\ hg-\kappa(v_i, {^g v_j})\ gh\\
&=
v_i \,  {}^g v_j\ gh
-q_{ij}\ \chi_i^{-1}(h)\, \chi_j(g)\
v_j\,  
{^{h} v_i}\ hg-\kappa(v_i, {^g v_j})\ gh\\
&=
\big(v_i\,  {^g v_j}
-q_{ij}
\ {^g v_j}\,  v_i-\kappa\
         (v_i, {^g v_j})\big)\ gh\\
&= \chi_j(g)\
\big(v_i \,  v_j
         -q_{ij}\  v_j\,  v_i
         -\kappa(v_i,  v_j)\big)\ gh.
\end{align*}
One may check that this isomorphism extends to an algebra isomorphism
$\U(L)\rightarrow \cH_{\bf q, \kappa}$:
$$
\U(L)
= T_{\kG}(V \tensor \kG)/ J
\ \overset{\cong}{\longrightarrow}\ 
T(V) \rtimes G
/I
=\cH_{\bf q, \kappa}
\, ,
$$
where $I$ is the ideal
$(v_i \,  v_j
         -q_{ij}\  v_j\,  v_i
         -\kappa(v_i,  v_j) :1\leq i,j\leq n)$.
Indeed, one may verify that the generators of $I$ correspond
to elements of $J$ under the inverse of the
isomorphism $T_{\kG}(V\ot \kG) \stackrel{\cong}{\longrightarrow}
T(V)\rtimes G$.
\end{proof}

\vspace{2ex}

Together, Theorem~\ref{thm:main1}
and Corollary~\ref{PBWconditions'} give
the following corollary.

\begin{cor}
\label{cor:main1fixedpointfree}
Let $G$ be a finite abelian group acting fixed point free and diagonally 
on $V=\k^n$ with respect to a $\k$-basis $v_1, \ldots, v_n$.
Suppose that $\Hqk$ defined by parameter
$\kappa:V\times V\rightarrow V\ot \kG$ 
is a quantum Drinfeld orbifold algebra. Then
the conclusion of Theorem~\ref{thm:main1} holds.
\end{cor}

\vspace{2ex}

\section{(Graded) Color Lie rings}
\label{gradedcolorLierings}

We now consider the case of (graded) color Lie rings.
We saw in the last section that quantum Drinfeld orbifold
algebras define ungraded color Lie rings when the vanishing 
condition~(\ref{vanishingcondition}) holds.
We show in this section that we obtain
(graded) color Lie rings when the vanishing condition holds
for {\em all} indices $i,j,k$, not just  distinct indices $i,j,k$.

Again, we consider a finite abelian group $G$ acting diagonally
on $V=\k^n$ with basis $v_1,\ldots, v_n$
and linear characters
$   \chi_i: G\rightarrow \k^* $ 
with
$  {}^g v_i = \chi_i(g) v_i 
  $ for $1\leq i\leq n$.

\vspace{2ex}

\begin{definition}\label{strongvanishingcondition}
    \em{We say a     quotient algebra $\Hqk$ as in~(\ref{eqn:Hqk}) satisfies
the {\em strong vanishing condition} if for all $g$ in $G$,
  $1\leq i,j,k \leq n$,
  and $1\leq r \leq n$,
$$ \,^gv_k = q_{ik}\, q_{jk}\, q_{kr}\ v_k$$
  whenever the coefficient of $v_r\ g$ is nonzero
  in $\kappa(v_i, v_j)\in V\ot \kG$.}
\end{definition}

\vspace{2ex}

One may check that Examples~\ref{ex:2} and~\ref{ex:3} satisfy
this strong vanishing condition, while Examples~\ref{ex:1} and~\ref{ex:4} do not. 

We now modify the grading by abelian group $A$
given in the last section to show that quantum Drinfeld orbifold algebras
satisfying the strong vanishing condition define
(graded) color Lie rings.
In Section~\ref{FirstMainThm}, we graded $L=V\ot\kG$
by $A=\ZZ^n\times G$
by setting
\begin{equation*}
|v_i|=(a_i , 1_G),
\quad
|g|=(0,g),
\quad\text{and}\quad
|v_i \ot g| = |v_i| |g|
\quad\text{ for all } g\in G\ .
\end{equation*}
Using this grading, we defined ungraded color Lie rings
from quantum Drinfeld orbifold algebras.
In order to obtain (graded) color Lie rings,
we replace $A$ by its quotient by a normal
subgroup $N$
in order to recapture
grading Condition~(i) in Definition~\ref{def:colorLie} of a color Lie ring.

Consider the dual basis
$\{(v_k\ot g)^*\}$ of $(V\ot \kG)^*=\Hom_{\k}(V\ot \kG, \k)$.
Define a subgroup $N$ of $A$ by a set of generators as follows:
\begin{equation*}
N := \langle\,
 |v_i| |v_j| |v_r|^{-1}|g|^{-1}: (v_r\ot g)^*\kappa(v_i, v_j)\neq 0
    \, \rangle .
\end{equation*}
The condition $(v_r\ot g)^* \kappa(v_i,v_j)\neq 0$
is precisely the 
condition $c^{ijg}_r\neq 0$ in the expansion
$\kappa(v_i,v_j)= \sum_{r,g} c^{ijg}_r v_r g$. 
We then obtain

\begin{thm}\label{DOAsAreColored} 
Fix a finite abelian group $G$ acting diagonally
on $V={\k}^n$. Let  $\cH_{{\bf q},\kappa}$
be a quantum Drinfeld orbifold algebra
defined by a parameter function $\kappa:V\times V\rightarrow V\ot \kG$ 
satisfying the strong vanishing condition~(\ref{strongvanishingcondition}).
Let $A=\ZZ^n\times G$ with subgroup $N$ defined above.
Then 
\begin{itemize}
\item[(a)]
There exists an antisymmetric bicharacter
$\epsilon: A/N\times A/N \rightarrow \k^*$ satisfying
\begin{equation*}
\hphantom{xxxx}   \epsilon(|v_i\ot g|,|v_j \ot h|) 
= q_{ij}\ \chi_j(g)\ \chi_i^{-1}(h)\, \ \ \text{ for }
g,h \text{ in } G, 1\leq i, j\leq n.
\end{equation*}
\item[(b)] 
The space $L=V\ot \kG$ is an $A/N$-graded algebra
under the bilinear operation 
$[\ ,\ ]: L \times L \rightarrow L$
defined by
\begin{equation*}
   [v_i\ot g,v_j\ot h] = \kappa(v_i, \, ^g v_j)\ g h
\quad\text{for}\ g,h\text{ in } G, 1\leq i,j\leq n\, .
\end{equation*}
\item[(c)]
The algebra
$L$ is a (graded) Yetter-Drinfeld $(A/N, \epsilon)$-color Lie ring
with universal enveloping algebra $\U(L)$ isomorphic to $\cH_{{\bf q},\kappa}$.
\end{itemize}
\end{thm}
\begin{proof}
The function $\epsilon$ given in the statement extends to an
antisymmetric bicharacter on $A$ as in the proof
of Theorem~\ref{thm:main1}, 
with 
$$\epsilon(|g|,|v_j|)=\chi_j(g)$$
for all $g\in G$ and $1\leq j\leq n$.
We first check that it is well-defined
on $A/N \times A/N$ 
and thus defines an antisymmetric bicharacter on $A/N$.
Suppose the coefficient of $v_r\ot g$ in $\kappa(v_i, v_j)$
is nonzero and consider any $v_k\ot h$ in $L$ with $h$ in $G$.
Note that $$\epsilon(|v_r|^{-1}|g|^{-1}, |v_k \ot h|)
=\epsilon(|v_r||g|, |v_k \ot h|)^{-1}\, .$$ 
Then
$$
\begin{aligned}
\epsilon(|v_i| |v_j|& |v_r|^{-1}|g|^{-1}, |v_k \ot h|)\\
&=
\epsilon(|v_i|, |v_k \ot h|)
\cdot\epsilon(|v_j|, |v_k \ot h|)
\cdot\epsilon(|v_r \ot g|, |v_k \ot h|)^{-1}\\
&=
q_{ik}\ \chi_i^{-1}(h)\
q_{jk}\ \chi_j^{-1}(h)\
q_{rk}^{-1}\ \chi_r(h)\ \chi_k(g)^{-1}\, \\
&=
\chi_r(h)\, \chi_i^{-1}(h)\chi_j^{-1}(h)\
q_{ik}\, q_{jk}\ 
q_{rk}^{-1}\ \chi_k(g)^{-1}\, .
\end{aligned}
$$
But Theorem~\ref{PBWconditions}(1) implies
that $\kappa$ is $G$-invariant and hence
$\chi_r(h)=\chi_i(h)\chi_j(h)$
by Lemma~\ref{Fixed point}.
By the strong vanishing condition~(\ref{strongvanishingcondition}),
the above expression is just $1$
and $\epsilon$ is well-defined on $A/N\times A/N$.

By Theorem~\ref{thm:main1}, $L$ 
is an ungraded color Lie ring.  
We check now that the first condition
in the definition of a color Lie ring holds as well.
Note that in $A$,
$$
|[v_i \ot h_i, v_j \ot h_j] |
=|\kappa(v_i, {}^{h_i} v_j)\ot h_i h_j|
=|\kappa(v_i, v_j)|\ |h_i h_j|\, .
$$
If the coefficient of $v_r\ot g$ in $\kappa(v_i, v_j)$
is nonzero,
then
\begin{equation}\label{ijrg-condition}
|v_i| |v_j|=|v_r||g|\quad\text{ in } A/N.
\end{equation}
Hence $|\kappa(v_i, v_j)|$ is $A/N$-homogeneous
with $|\kappa(v_i, v_j)|=|v_i| | v_j|$,
and the color Lie bracket $[\ ,\ ]$ is $A/N$-graded.
\end{proof}

From now on, we may replace $A = \Z^n\times G$ with $A/N$
where convenient, in order to work with color Lie rings instead 
of ungraded color Lie rings.

The following corollary gives some alternate conditions under which
one obtains the same conclusion as in the theorem. 

\begin{cor}
Fix a finite abelian group $G$ acting diagonally
on $V={\k}^n$ and fixed point free. 
Let  $\cH_{{\bf q},\kappa}$
be a quantum Drinfeld orbifold algebra
defined by a parameter function $\kappa:V\times V\rightarrow V\ot \kG$ 
such that for all $1\leq i,j,r \leq n$,
$$\, ^gv_i = q_{ji}\, q_{ir}\ v_i
\quad\text{ whenever }\quad 
(v_r\ot g)^* \kappa(v_i,v_j)\neq 0\, .$$
Then the conclusion of Theorem~\ref{DOAsAreColored} holds.
\end{cor}
\begin{proof}
  We argue that the given hypothesis implies the strong vanishing
  condition~(\ref{strongvanishingcondition}) needed for Theorem~\ref{DOAsAreColored}.
  By Theorem~\ref{PBWconditions} and Lemma~\ref{Fixed point},
  the algebra $\cH_{{\bf q},\kappa}$ satisfies
  vanishing condition~(\ref{vanishingcondition}).  Hence
  we need only check
  the vanishing condition when indices coincide.
  If $i=j$, then $\kappa(v_i, v_i)=0$ so there is nothing to check.
    The condition for $k=i$ is the assumption stated.
    Since $\kappa(v_i,v_j)=-q_{ij}\, \kappa(v_j, v_i)$,
    the condition for $k=i$ implies the condition for $k=j$.
\end{proof}

\section{Braided Lie algebras and Hopf algebras}\label{sec:cat-th}

In this section, we view color Lie rings as Lie algebras in a particular category and derive a braided Hopf algebra structure
for certain quantum Drinfeld orbifold algebras. 

\subsection*{Lie algebras in a symmetric monoidal category.}
  For any commutative unital ring 
  $\mathbb{K}$ and $\mathbb{K}$-linear symmetric monoidal category $\C$,
  one can define a Lie algebra as follows
  (see, e.g.,~\cite{Kharchenko-book,KharchenkoShestakov,Pareigis}).
  Denote the monoidal product in $\C$ by $\ot$ and the braiding by $\tau$. A
  {\em Lie algebra} in $\C$ is an object $L$ in $\C$ together with a morphism $[ \;, \;] : L \ot L \to L$ satisfying antisymmetry and  the Jacobi identity, i.e.,
\begin{gather}
[ \;, \;]+[\;, \;] \circ \tau=0 , \label{AntiSymMonCat} \\
[ \;, [ \;, \;]]+[ \;, [ \;, \;]] \circ (\id \otimes \tau) \circ (\tau \otimes \id) + [ \;, [ \;, \;]]\circ (\tau \otimes \id) \circ (\id \otimes \tau)=0 . \label{JacobiMonCat}
\end{gather}

\subsection*{The category for color Lie rings}

Some of our results in this paper can be phrased alternatively in
the language of category theory.  We again fix
an abelian group $A$ and ring $R$ and take 
$\C$ to be the category of $A$-graded $R$-bimodules with
monoidal product $\otimes_R$
and graded morphisms. 
Any bicharacter $\epsilon$ on $A$ gives rise to 
a braiding $\tau$ on $\C$ in the following way:
For objects $V,W$ of $\C$, define
$\tau = \tau_{V,W}: V\ot_R W\rightarrow W\ot_R V$
by
\begin{equation}\label{eqn:braid-def}
   \tau(v\ot_R w) = 
   \epsilon(|v|, |w|)\ w\ot_R v
\end{equation}
   for all $A$-homogeneous $v\in V$, $w\in W$. Then the bicharacter
   condition on $\epsilon$ is equivalent to the braiding condition on $\tau$,
   that is, $\tau$ satisfies the hexagon identities
\[ \tau_{U, V\ot_R W} = (1_V\ot_R\tau_{U,W})(\tau_{U,V}\ot_R 1_W)
  \ \ \ \mbox{ and } \ \ \
   \tau_{U\ot_R V, W} = (\tau_{U,W}\ot_R 1_V)(1_U\ot_R \tau_{V,W})
\]
for all $U,V,W$ in $\C$. Since $\epsilon$ is antisymmetric, $\tau ^2=1$, and $\C$ is a symmetric monoidal category.

\subsection*{Color Lie rings as Lie algebras}
An $(A,\epsilon)$-color Lie ring over $R$
is a Lie algebra in this category $\C$.
Indeed,
Definition~\ref{def:colorLie}(i) states that $[ \;, \;]$ is a graded map
and hence a graded morphism in $\C$
(consider the product $L \times L$ as an object in $\C$ with grading
$(L \times L)_c=\oplus_{a,b:\, ab=c}\, L_a \times L_b$).
Definition~\ref{def:colorLie}(ii) and (iii) are equivalent to \eqref{AntiSymMonCat} and~\eqref{JacobiMonCat}, respectively. A calculation shows that Definition~\ref{def:colorLie}(i) also
implies that the bracket $[ \ , \ ]$ is compatible with the
braiding $\tau$ in the sense that
\begin{equation}\label{eqn:compatibility-A}
  \tau ( 1\ot_R [ \ , \ ]) = ( [ \ , \ ] \ot_R 1)(1\ot_R \tau )(\tau\ot_R 1) .
\end{equation}
A second compatibility condition
\[
   \tau ( [ \ , \ ] \ot_R 1) = (1\ot_R [ \ , \ ]) (\tau\ot_R 1)(1\ot_R \tau)
\]
follows from the first since $\tau^2=1$; just multiply both sides of the first
condition by $\tau$ on the left and by $(\tau\ot_R 1)(1\ot_R\tau)$ on the right.
(More generally, these two compatibility conditions are in fact
necessary conditions to have a Lie algebra in a symmetric
monoidal category: The braiding $\tau$ by definition consists of 
functorial isomorphisms and thus must satisfy commutative diagrams
corresponding to morphisms in the category. 
For the particular morphism given by the bracket
operation, the compatibility conditions as given above are
equivalent to commuting diagrams arising from morphisms from
three copies of the Lie algebra to two.)

\subsection*{Universal enveloping algebras in the category.}
We next observe that color universal enveloping algebras are  universal enveloping algebras in the specific category above.
One can define the notion of {\em universal enveloping algebra} of a
Lie algebra in any symmetric monoidal category $\C$ via a universal property
as follows 
(see, e.g.,~\cite{Kharchenko-book,KharchenkoShestakov,Pareigis}).
We first define an associative algebra in $\C$ to be an object $B$ together with a
morphism $m : B \otimes B \to B$ satisfying
$ m \circ (m \otimes 1)=m \circ (1 \otimes m)$.
A calculation shows that $B$ defines a Lie algebra in $\C$, denoted
Lie$(B)$, by setting $[\;,\;] :=m -m \circ \tau$.
For a Lie algebra $L$ in $\C$, the universal enveloping algebra $\U(L)$ is an associative algebra together with an injective map $i : L \hookrightarrow \U(L)$ satisfying the universal property that for any associative algebra $B$ in $\C$ and Lie algebra map $\phi : L \to \text{Lie}(B)$, there exists a unique map of associative algebras $\hat{\phi} : \U(L) \to B$ such that $\phi=\hat{\phi} \circ i$.
Note that such an associative algebra $\U(L)$ may not exist in general. If it does exist, then it is unique.
In the case where $\C$ is the category of $A$-graded $R$-bimodules described above,
the universal enveloping algebra exists for all Lie algebras in $\C$,
and it is given explicitly in Definition~\ref{def of env alg}.

\subsection*{Hopf algebra structures}

We now point out
a Hopf algebra structure on color universal enveloping algebras
defined via the category setting.
Generally, a {\em Hopf algebra} in a $\k$-linear symmetric monoidal category $\C$
is an object in $\C$ with defining maps
(unit, multiplication, counit, comultiplication, antipode) all morphisms
in the category satisfying the standard Hopf algebra properties.
We sometimes simply speak of a {\em braided Hopf algebra} when it is clear
from context which braiding and category are intended.

In certain categories, universal enveloping algebras always exhibit
the structure of a braided Hopf algebra.
Let $\C$ be a $\mathbb{K}$-linear symmetric monoidal category for which the universal enveloping algebra exists for all Lie algebras in $\C$ and the assignment $L \mapsto \U(L)$ is functorial.
Then $\U(L)$ has a braided Hopf algebra structure, i.e., it has a counit map
$\epsilon : \U(L) \to \U(0)={\bf 1}$
induced by the Lie algebra map $L \to 0$,
a comultiplication $\Delta: \U(L) \to \U(L \times L) \iso \U(L) \ot \U(L)$ induced by the diagonal map $L \to L \times L$, and an antipode $S : \U(L) \to \U(L)$ induced by $-1: L \to L$.
These maps also satisfy some properties with respect to the braiding;
e.g., see~\cite{Kharchenko-book,KharchenkoShestakov}.

The morphisms giving color universal enveloping algebras the structure
of  braided Hopf algebras are given explicitly in the next proposition,
a consequence of~\cite[Section~4]{KharchenkoShestakov}.
See also~\cite[Proposition~2.7]{FischmanMontgomery} for a special case. 

\begin{prop}\label{prop:Hopf}
Let $L$ be an $(A,\epsilon)$-color Lie ring over $R$,
$\U(L)$ its universal enveloping algebra, and $\C$ the category
of $A$-graded $R$-bimodules with monoidal product $\ot_R$.  
Then  $\U(L)$ is a Hopf algebra in $\C$:
\begin{enumerate}
\item The tensor algebra $T_R(L)$ of $L$ 
 is a Hopf algebra in $\C$ with coproduct, counit, and antipode defined by
\[ \Delta(l)=l \otimes_R 1+1 \otimes_R l, \qquad \epsilon(l)=0, \qquad S(l)=-l, 
\qquad \mbox{ for all } l \in L. \]
\item The ideal $J$ is a Hopf ideal in $\C$, and consequently $\U(L)$ is a Hopf algebra
in $\C$.
\end{enumerate}
\end{prop}

We wish to conclude that there is a braided Hopf structure on
the quantum Drinfeld orbifold algebras $\Hqk$ of Theorem~\ref{DOAsAreColored}.
We first make an observation
interesting in its own right:
The strong vanishing condition assumed
in Theorem~\ref{DOAsAreColored}
is equivalent to compatibility of the bracket operation
$[ \ , \ ]$ with the braiding $\tau$.
We give a direct proof for interest, although we use similar arguments
in other sections.

\begin{prop}\label{prop:svc}
Let $G$ be a finite group acting diagonally on $V= \k^n$.
Let $\Hqk$ be a quantum Drinfeld orbifold algebra satisfying
the vanishing condition~(\ref{vanishingcondition}).
The strong vanishing condition~(\ref{strongvanishingcondition}) is equivalent
to the compatibility~(\ref{eqn:compatibility-A}) 
of the braiding $\tau$ defined by~(\ref{eqn:braid-def})
with the bracket $[ \ , \ ]$ on $L=V\ot \kG$ given in Theorem~\ref{thm:main1}(b).
\end{prop}

\begin{proof}
The compatibility condition~(\ref{eqn:compatibility-A}),
applied to $v_k\ot_{\kG} v_i\ot_{\kG} v_j$, may be written
\[
   \tau(v_k\ot _{\kG} \kappa(v_i,v_j)) = ({[} \ , \ {]}\ot_{\kG} 1) (1\ot_{\kG}\tau)
    (\tau(v_k,v_i)\ot_{\kG} v_j),
\]
and applying the definition~(\ref{eqn:braid-def}) of $\tau$, this is
equivalent to 
\[
  \epsilon ( |v_k|, |\kappa(v_i,v_j)|)  =
   \epsilon (|v_k|,|v_i|)\, \epsilon(|v_k|,|v_j|) .
\]
In turn, this equation may be rewritten as  
\[
   q_{kr}\, \chi_k^{-1}(g) = q_{ki}\, q_{kj}
\]
for all $r,g$ for which $c^{ijg}_r\neq 0$.
This is required to hold
for all $i,j,k$, and that is precisely the strong
vanishing condition~(\ref{strongvanishingcondition}).
We note that the compatibility condition~(\ref{eqn:compatibility-A}) applied
more generally to elements $v_k g\ot_{\kG} v_i h \ot_{\kG} v_j l$
follows from this case $g=h=l=1$ by the definitions of
$\tau$, $\kappa$, and $ [ \ , \ ]$.
\end{proof}

The following corollary is an immediate consequence of 
Theorem~\ref{DOAsAreColored} and  Proposition~\ref{prop:Hopf}.

\begin{cor}\label{cor:Hopf-alt}
Let $G$ be a finite abelian group acting diagonally on the vector space $V$.
Let $\Hqk$ be a corresponding quantum Drinfeld orbifold algebra for which
the strong vanishing condition~(\ref{strongvanishingcondition}) holds.
Then $\Hqk$ is a braided Hopf algebra.
\end{cor}

\begin{proof}
By Theorem~\ref{DOAsAreColored}(c), $L=V\ot \kG$ is an $(A/N,\epsilon)$-color
Lie ring with universal enveloping algebra ${\mathcal{U}}(L)\cong \Hqk$.
Let $\C$ be the category of $A/N$-graded $\kG$-bimodules with 
monoidal product $\ot_{\kG}$, graded morphisms,
and braiding $\tau$ given by~(\ref{eqn:braid-def}).
By Proposition~\ref{prop:Hopf}, ${\mathcal{U}}(L)\cong \Hqk$
is a Hopf algebra in $\C$, that is, a braided Hopf algebra.
\end{proof}

\section{Color universal enveloping algebras 
as \\
$ \ \ {}_{}\ \ \ $
quantum Drinfeld orbifold algebras}
\label{SecondMainThm}
In this section, we 
determine those Yetter-Drinfeld color
Lie rings that arise from quantum Drinfeld orbifold algebras
and
establish a converse to Theorem~\ref{thm:main1}
and Theorem~\ref{DOAsAreColored}.
But first we discuss positive and negative parts
of color Lie rings.

\subsection*{Positive and negative parts of color Lie rings}
Recall that in any $(A,\epsilon)$-color Lie ring, 
$\epsilon(|x|, |x|)=\pm 1$ for all $A$-homogeneous $x$, 
introducing a $\ZZ/2\ZZ$-grading.
In fact,
$$
[x,x]=0
\quad\text{ or }\quad 
\epsilon(|x|, |x|) = -1
\quad\text{ for all $A$-homogeneous } x.
$$ 
For a color Lie ring $L=V\ot \kG$ with $G$ a finite group acting on $V$, we
define the {\em positive} and
{\em negative parts} of $V$
(just as for color Lie algebras),
$$
\begin{aligned}
V_-&=\{A\text{-homogeneous}\ v\in V: \epsilon(|v|,|v|)=-1\}
\quad\text{ and}\\
V_+&=\{A\text{-homogeneous}\ v\in V: \epsilon(|v|,|v|)=+1\},
\end{aligned}
$$
and define
$$L_+=V_+\ot \kG
\quad\text{ and }\quad
L_-=V_- \ot \kG
$$
so that $L=L_-\oplus L_+$.
We say a color Lie ring $L=V\ot \kG$ has
{\em purely positive part} when $L=L_+$.
We will see that Yetter-Drinfeld color Lie rings
with purely positive part arise
from quantum Drinfeld orbifold algebras.
(See Example~\ref{superalg} for a
color Lie ring with $x$ 
satisfing $[x, x]\neq 0$.)

Note that
for a Yetter-Drinfeld color Lie ring
$L=V\ot \kG$,
the set
$$\{q_{ij}=\epsilon(|v_i|, |v_j|)\}$$
may fail to be a quantum system of parameters.
Indeed, if
$[v_i\ot 1_G, v_i\ot 1_G]\neq 0$
for some $i$, then 
$q_{ii}=\epsilon(|v_i|, |v_i|)=-1$.
In fact, if
$[v_i\ot 1_G, v_i\ot 1_G]= 0$
for some $i$ but
$q_{ii}=\epsilon(|v_i|, |v_i|)=-1$,
the element $v_i^2$ is nilpotent
in the universal enveloping algebra
$\mathcal{U}(L)$;
see Example~\ref{superalg}.
In such cases, the set $\{q_{ij}\}$
could be used to define
truncated quantum Drinfeld Hecke algebras as in 
Grimley and Uhl~\cite{GrimleyUhl}.

The next proposition shows how the last
condition in the definition
of a Yetter-Drinfeld color Lie ring 
arises from the positive and negative
parts.

\begin{prop}\label{LittleLemma}
    Suppose
  that
  $L=V\ot \kG$ is an $(A,\epsilon)$-color
  Lie ring over $\kG$ with $A$-grading on $L$
  induced from gradings on $V$ and $G$.
  Assume $G$ acts diagonally
  on $V\cong \k^n$ with respect to a basis
  $v_1, \ldots, v_n$.
  Then 
for all $1\leq i\leq n$ and $g$ in $G$,
\begin{equation*}
[v_i\ot 1_G, v_i\ot g]=0 
\quad\text{ or }\quad
^g v_i
   = - \epsilon(|v_i|, |v_i|)\ \epsilon(|g|,|v_i|) \ v_i \, .
\end{equation*}
If, in addition, $L$ is Yetter-Drinfeld
    with purely positive part, then
$[v_i\ot 1_G, v_i\ot g]=0$ for all $g$ in $G$.
\end{prop}
\begin{proof} 
  Suppose $\, ^g v_i = \chi_i(g)\ v_i$.
  As the bracket is $\kG$-balanced 
and $|v_i\ot g|=|v_i| |g|$, 
$$
\begin{aligned} 
\ [ v_i \ot 1_G, v_i \ot g]
&= [v_i\ot 1_G,\, g( {}^{g^{-1}}\hspace{-.5ex} v_i \ot 1_G)]
= [v_i\ot g, {}^{g^{-1}}\hspace{-.5ex}v_i \ot 1_G]\\
&= \chi_i(g)^{-1}\ [v_i\ot g, v_i \ot 1_G]\, \\
&= 
\chi_i(g)^{-1}\ 
(-1)\ \epsilon(|v_i\ot g|, |v_i\ot 1_G|)\
[v_i\ot 1_G, v_i \ot g]\, .
\end{aligned}
$$
Then for nonzero $[ v_i \ot 1_G, v_i \ot g]$,
since $\epsilon$ is a bicharacter and $|1_G|=1_A$, 
\begin{equation*}
\begin{aligned}
-\chi_i(g) 
&= \epsilon(|v_i\ot g|, |v_i\ot 1_G|)
= \epsilon\big(|v_i| |g|, |v_i| |1_G| \big)\\
&= \epsilon\big(|v_i| |g|, |v_i| \big)
= \epsilon(|v_i|, |v_i|)\ 
\epsilon(|g|, |v_i|)\, ,
\end{aligned}
\end{equation*}
establishing the first claim.
If $L$ is Yetter-Drinfeld
    with purely positive part, then
    $$ \epsilon(|g|, |v_i|) \ v_i
= \, ^gv_i
=-\epsilon(|v_i|, |v_i|)\,
\epsilon(|g|, |v_i|)\ v_i
=-\epsilon(|g|, |v_i|)
\ v_i\, ,
$$
from which the second claim follows.
\end{proof}

\subsection*{Yetter-Drinfeld color Lie rings with purely positive part}
We show now that every Yetter-Drinfeld color Lie ring over $\kG$
with purely positive part
corresponds to a quantum Drinfeld orbifold algebra.
We consider an arbitrary (graded) Yetter-Drinfeld color Lie
ring defined by some abelian grading group $A$ and
bicharacter $\epsilon$.

\begin{thm}\label{thm:main2}
Let $L=V\ot \kG$ be a Yetter-Drinfeld $(A,\epsilon)$-color Lie ring
over $\kG$ for some group $G$ with purely positive part.
Then there exist parameters $\kappa: V \times V \rightarrow V \otimes \kG$
and ${\bf q}$
so that the universal enveloping algebra
$\mathcal{U}(L)$
is
isomorphic to the quantum
Drinfeld orbifold algebra 
$\cH_{{\bf q},\kappa}$.
\end{thm}
\begin{proof}
  Let $v_1,\ldots, v_n$ be an $A$-homogeneous basis of $V$
  and recall
  that $G$ must act diagonally with respect to this basis
  since $L$ is Yetter-Drinfeld.
  Define a parameter function $\kappa: V\ot V \rightarrow V\otimes \kG$
by
$$
\kappa(v_i, v_j)=[v_i\ot 1_G, v_j\ot 1_G]
\quad\text{ for all } 1\leq i,j \leq n\, 
$$
and a quantum system of parameters ${\bf q}=\{ q_{ij}\}$
(using that  $L$ is purely positive) by
$$q_{ij}=\epsilon(|v_i|, |v_j|)
\quad\text{ for all } 1\leq i , j \leq n\, .
$$
Then
$\kappa$ is quantum symmetric as $\epsilon$
is antisymmetric.
We show that $\cH_{{\bf q},\kappa}$ is a quantum Drinfeld orbifold algebra by checking
Conditions~(1$'$),~(2$'$),~and~(3$'$) of Corollary~\ref{PBWconditions'}.

Condition~(3$'$) follows from the $\epsilon$-Jacobi identity
as $L$ is purely positive.

Condition~(1$'$), the $G$-invariance of $\kappa$, follows from the fact
that $[\ ,\  ]$ is $\kG$-bilinear and $\kG$-balanced:
$$\begin{aligned}
\kappa( ^gv_i,\ ^g v_j)
&=& &[\, ^gv_i\ot gg^{-1},\  ^gv_j\ot 1_G]&
&=& &[g(v_i\ot 1_G)g^{-1},\ ^gv_j\ot 1_G]&\\
&=& &g\, [v_i\ot 1_G, g^{-1}(^gv_j\ot 1_G)]&
&=& &g\, [v_i\ot 1_G, v_j\ot g^{-1}]&\\
&=& &g\, [v_i\ot 1_G, (v_j\ot 1_G)\, g^{-1}]&
&=& &g\, [v_i\ot 1_G, v_j\ot 1_G]\, g^{-1}.&
\end{aligned}
$$
But this last expression is $^g(\kappa(v_i, v_j))$
(with $G$-action on $\kappa$
induced from the adjoint action~(\ref{adjointaction})), and thus $\kappa$ is $G$-invariant.

We argue that Condition~(2$'$) follows from the fact that
the color Lie bracket is $A$-graded.
Suppose the coefficient of $v_r\ot g$ in $\kappa(v_i, v_j)$
is nonzero (so $i\neq j$) for some $1\leq r\leq n$ and $g\in G$.  Then 
as $|v_i| |v_j|=|\kappa(v_i, v_j)|$ in $A$,
we have $|v_i| |v_j| = |v_r\ot g|$.
Condition~(2$'$) then follows since $\epsilon$ is a bicharacter, $|1_G|=1_A$,
and $L$ is Yetter-Drinfeld:
$$
\begin{aligned}
q_{rk}\ \chi_k(g) 
&= \epsilon(|v_r|, |v_k|)\ \epsilon(|g|, |v_k|)
=\epsilon(|v_r|  |g|, |v_k|  |1_G|)
=\epsilon(|v_r\ot g|, |v_k \ot 1|)\\
&=\epsilon(|v_i| |v_j|, |v_k|)
=\epsilon(|v_i|, |v_k|)\ \epsilon( |v_j|, |v_k|)
=q_{ik}\ q_{jk}\, .
\end{aligned}
$$
\end{proof}

By the definition of the quantum Drinfeld orbifold algebra $\Hqk$,
we immediately conclude that the universal enveloping algebra $\U(L)$
of a Yetter-Drinfeld color Lie ring $L$
is a PBW deformation of $S_{\bf{q}}(V)\rtimes G$:

\begin{cor}\label{lastcor}
  Let $L=V\ot \kG$ be a Yetter-Drinfeld $(A,\epsilon)$-color Lie ring
  with purely positive part.  Then
$\U(L)$ has the PBW property.
\end{cor}

We collect our main results from
this section and last,
making precise the connection between the 
universal enveloping algebras of Yetter-Drinfeld color Lie rings and  
quantum Drinfeld orbifold algebras.  
Compare with~\cite[Theorem~3.9]{ShroffWitherspoon} in the special case $G=1$. 
Recall that in any color Lie ring $L$,
$\epsilon(|x|, |x|)=1$ 
forces $[x, x]=0$ for $x$ in $L$.
The following statement is now a consequence of
Theorem~\ref{DOAsAreColored} 
and Theorem~\ref{thm:main2}. 
\begin{thm}\label{thm:main-recap}
Let $G$ be a finite abelian group acting diagonally
on $V=\k^n$.
\begin{itemize}
\item[(i)]
The universal enveloping algebra
of any Yetter-Drinfeld color Lie ring
with purely positive part
is isomorphic to a
quantum Drinfeld orbifold algebra $\cH_{{\bf q},\kappa}$.
\vspace{1ex}
\item[(ii)]
Any quantum Drinfeld orbifold algebra $\cH_{{\bf q},\kappa}$
with $\kappa: V\times V\rightarrow V\ot  \k G$ 
satisfying the strong vanishing condition~(\ref{strongvanishingcondition})
is isomorphic to the universal enveloping algebra of some
Yetter-Drinfeld color Lie ring
with purely positive part.
\end{itemize}
\end{thm}


\vspace{2ex}

\section*{Acknowledgments}
This project began at the 
Women in Noncommutative Algebra and Representation Theory (WINART)
Workshop at Banff International Research Station (BIRS) in March 2016.
We thank BIRS for providing a very productive work environment. 
Kanstrup gratefully acknowledges the support of the Max Planck
Institute for Mathematics, Bonn.
Kirkman was partially supported by Simons Foundation grant
208314, Shepler by Simons Foundation grant 429539, 
and Witherspoon by NSF grants DMS-1401016 and DMS-1665286.



\begin{thebibliography}{999}

\bibitem{CBH} 
W.\ Crawley-Boevey and M.\ P.\ Holland, ``Noncommutative deformations of
Kleinian singularities,'' Duke Math.\ J.\ 92 (1998), no.\ 3, 605--635. 

\bibitem{EtingofGinzburg} P.\ Etingof and V.\ Ginzburg,
``Symplectic reflections algebras, Calogero-Moser space, and deformed
  Harish-Chandra homomorphism,'' Invent.\ Math.\ 147 (2002), no.\ 2,
243--348. 

\bibitem{FischmanMontgomery}
D.\ Fischman and S.\ Montgomery, ``A Schur double centralizer theorem for cotriangular Hopf algebras and generalized
Lie algebras,'' J.\ Algebra 168 (1994), 594--614.

\bibitem{GrimleyUhl}
  L.\ Grimley, C.\ Uhl, ``Truncated quantum Drinfeld Hecke algebras
and Hochschild cohomology,'' 
arXiv:1711.03216.
  
\bibitem{Gurevich}
D.\ Gurevich, ``Generalized translation operators on Lie groups,''
Soviet J.\ Contemporary Math.\ Anal.\ 18 (1983), 57--70. (Izvestiya 
Akademii Nauk Armyanskoi SSR Matematica 18, N4 (1983), 305--317.) 

\bibitem{Kharchenko07} 
V.\ K.\ Kharchenko, ``Connected braided Hopf algebras,'' 
J.\ Algebra 307 (2007), 24--48. 

\bibitem{Kharchenko-book}
V.\ K.\ Kharchenko,
{\em Quantum Lie Theory:
A Multilinear Approach},
Springer, 2015. 

\bibitem{KharchenkoShestakov}
V.\ K.\ Kharchenko and I.\ P.\ Shestakov, ``Generalizations of Lie Algebras,'' Adv. Appl. Clifford Algebras 22 (2012), 721--743.

\bibitem{LevandovskyyShepler} 
V.\ Levandovskyy and A.V.\ Shepler,
``Quantum Drinfeld Hecke algebras,"
Canad.\ J.\ Math. 66 (2014), no. 4, 874--901.

\bibitem{Lusztig} G.\ Lusztig,
``Affine Hecke algebras and their graded version,''
J.\ Amer.\ Math.\ Soc.\ 2 (1989), no.~3, 599--635. 


\bibitem{Pareigis} B.\ Pareigis,
``On Lie algebras in braided categories,''
in: Quantum Groups and Quantum Spaces, 
Banach Center Publ.\ 40, Polish Acad.\ Sci., Warsaw, 1997, 139--158. 

\bibitem{Shakalli} J.\ Shakalli,
``Deformations of quantum symmetric algebras extended by groups,'' 
J.\ Algebra 370 (2012), 79--99. 

\bibitem{DOA}
A.\ Shepler and S.\ Witherspoon,
``Drinfeld orbifold algebras,"
Pacific J.\ Math. 259 (2012), no. 1, 161--193.




\bibitem{Shroff}
P.\ Shroff,
``Quantum Drinfeld orbifold algebras,'' 
Communications in Algebra 43 (2015), 1563--1570.

\bibitem{ShroffWitherspoon}
P.\ Shroff and S. Witherspoon,
``PBW deformations of quantum symmetric algebras and their group extensions,"
 J.\ Algebra and Its Applications
15 (2016), no.\ 3, 15 pp. 




\end{thebibliography}
\end{document}